\documentclass[doublecolumn,letterpaper,10 pt, conference]{ieeeconf} 
\usepackage{times} % assumes new font selection scheme installed
\usepackage{amsmath} % assumes amsmath package installed
\usepackage{amssymb,amsfonts}  % assumes amsmath package installed
\usepackage{latexsym,theorem,epsfig}
\usepackage{graphicx} 
\usepackage{dsfont}
\usepackage{mathtools}
\usepackage{subfigure}
\usepackage{lscape}
\usepackage{color}
\usepackage{float}
\usepackage{enumerate}
 \usepackage[usenames,dvipsnames]{pstricks}
 \usepackage{epsfig}
 \usepackage{pst-grad} % For gradients
 \usepackage{pst-plot} % For axes

\newtheorem{theorem}{Theorem}

{\theorembodyfont{\rmfamily} }
\newtheorem{lemma}[theorem]{Lemma}

\newtheorem{proposition}[theorem]{Proposition}

\newtheorem{rem}{Remark}

\allowdisplaybreaks
\IEEEoverridecommandlockouts
\DeclareMathOperator*{\argmax}{arg\,max}
\title{Network Independent Rates in Distributed Learning}

\author{Angelia Nedi\'{c}, Alex Olshevsky and C\'{e}sar A. Uribe
\thanks{The authors are with the Coordinated Science Laboratory, University of Illinois, 1308 West Main Street, Urbana, IL 61801, USA, \{angelia,aolshev2,cauribe2\}@illinois.edu.
	This research is supported partially by the National Science Foundation under
	grants no.\ CCF 11-11342 and no.\ CMMI-1463262 and by the Office of Naval Research under grant
	no.\ N00014-12-1-0998.}  
}

\begin{document}
\maketitle
\begin{abstract}
We propose a new belief update rule for \textit{Distributed Non-Bayesian learning} in \textit{time-varying directed} graphs, where a group of agents tries to collectively identify a hypothesis that best describes a sequence of observed data. We show that the proposed update rule, inspired by the Push-Sum algorithm, is consistent; moreover we provide an explicit characterization of its convergence rate. Our main result states that, after a transient time, all agents will concentrate their beliefs at a network independent rate. Network independent rates were not available for other consensus based distributed learning algorithms. 
\end{abstract}

\section{Introduction}

 Numerous engineered and natural systems can be modeled as a group of agents (people, robots, sensors, etc.) interacting with each other as well as with the environment where they are located. The rules agents follow as well as the structure of the environment determines their ability to make decisions in a distributed manner. In this paper, we study the distributed non-Bayesian learning model where a group of agents tries to ``learn" a hypothesis (from a parametrized family) that best explains some observed data \cite{jad12,sha13,sha14,lal14b,ned14,ned15,sha15,rah15}. Observations are realizations of a vector random variable with unknown distribution. Moreover, learning should be done in a distributed manner where each agent only access an specific entry of the realizations vector without the involvement of any centralized coordination. They are also allowed to interact with other agents on a network modeled as a \textit{time-varying directed} graph. 
 
Non-Bayesian learning has been previously studied in the context of social learning with boundedly rational agents \cite{tah09,gol10,ace08,jad12}. In contrast with fully rational agents \cite{ace11,gal03,mos14,mos15}, boundedly rational agents fail to aggregate information in a fully Bayesian manner \cite{gol10}. Agents repeatedly communicate with others and use naive approaches (e.g. weighted averages) to aggregate information. Initial results proposed distributed non-Bayesian ways to aggregate beliefs following DeGroot model for opinion aggregation \cite{deg74,jad12}. Similar approaches has been used in distributed estimation \cite{rab04,rah10}, where belief propagation as been shown effective in solving hypothesis testing problem in a distributed manner \cite{ala04,sal06}. Others authors used optimization methods such as distributed dual averaging approach proposed in \cite{nes09} to solve the same parameter estimation problem showing asymptotic exponential convergence rates in terms of the learning structure of the problem \cite{sha13, lal14b}. Non-asymptotic rates have been recently derived for fixed graphs \cite{sha14} and time-varying directed graphs \cite{ned14}. 

Previous results assume the existence of a ``true state" of the world, such that the unknown distribution perfectly matches one element in the parametrized family of hypothesis. The authors in \cite{ned14} extended this approach to allow the non-realizable case where the distribution of the observations need not be a member of the family of hypothesis. Nevertheless it was assumed that the optimal hypothesis was an element of the intersection of the locally optimal hypothesis from prospective of individual agents. Later in \cite{ned15}, the authors introduced the concept of \textit{conflicting hypothesis} where locally optimal hypothesis from independent social clicks or subsets of agents need not intersect. They showed that a distributed non-Bayesian learning approach will generate all the agents to ``learn"  the hypothesis that best explains the group observations (i.e. largest \textit{group confidence}) even if it was not locally optimal. 

Conflicting hypothesis can be interpreted as different social clicks having different optimal hypothesis; which could represent faulty sensors, or malicious agents trying to affect the network. In general, agents with higher connectivity will have more influence in defining the hypothesis with largest group confidence. Thus, a faulty sensor with good connectivity can severely hinder the performance of the estimation process. In this context, in \cite{gol10} the authors defined a society as ``wise" if the influence of the most influential agents vanishes with the size of the network, or similarly for graphs of fixed size, the there is some balancedness in network. In the case of static graphs, knowledge about the topology of the network can be used to overcome network imbalance, but this introduces additional requirements and limits the ad-hoc nature of a distributed solution. When the network is not static, the connectivity of the agents changes with time and thus its influence, introducing variability in the hypotheses group confidence. The authors in \cite{ned15} avoided this situation by assuming the network had some ``balanced" properties that assured network independence in the group confidence values. 

Our contribution is three fold. First, we propose a new update rule that extends the recent result of \cite{ned15} to time-varying directed graphs, where the optimal hypothesis have the highest group confidence regardless agents influence. Second, we show that this rule converges at a geometric rate that is network independent and achieves long-term balancing of the graph sequence. This improves on previous results on distributed learning rules for time-varying directed graphs, \cite{ned14}, where the learning rate depends on the network balance. Finally we develop a general framework for deriving non-asymptotic convergence rates for update rules that can be expressed as log-linear functions.

This paper is organized as follows: Section \ref{problem} describes the distributed learning problem and the proposed belief update rule; it also states the main result that describes the geometric, balanced and network independent convergence rate. Section \ref{rates} presents the detailed proof of the main result. Conclusions and future work are presented in Section \ref{conclusions}.

\textit{Notation}: Random variables are represented as upper case letters, i.e. $X$, whereas their realizations as its corresponding lower case, i.e. $x$. Time indexes are indicated by subscripts and make use of the letter $k$. Agent indexes are represented as superscripts and use the letters $i$ or $j$. Bold letters indicate vectors, where $[\boldsymbol{X}_k]_i = X_k^i$. The $i$-th row and $j$-th column entry of a matrix $A$ is denoted as $[A]_{ij}$.

\section{Problem statement and Main Result}\label{problem}

Consider a group of $n$ agents, $V = (1,2,\hdots,n)$, each observing realizations of independent processes at each time step $k=1,2,\hdots$. Agent $i$ observes realizations of a sequence of stationary independent, identically distributed random variables  $\{S_k^i\}$ with unknown distribution $f^i\left(\cdot\right)$. Staking all the random variables at time $k$ generates a single vector $\boldsymbol{S}_k$ distributed as $\boldsymbol{f}= \prod\limits_{i=1}^{n}f^i$. 

The group objective is to collectively agree on a parameter, that describes a probability distribution from a prespecified family, closest to the true distribution of the observations. Each agent has a family of parametrized distributions $(l^i(\cdot|\theta))$ with parameter $\theta \in \Theta$. The set $\Theta$ is common to all agents and it is assumed finite; it can be understood as a set of parameters that characterizes possible probability distribution for the group observations. Probability distributions over the set $\Theta$ are refereed as beliefs. 

Under this setup, the group of agents collectively tries to solve the following optimization problem
\begin{align}\label{main_problem}
\Theta^* = \argmax\limits_{\theta \in \Theta} \mathsf{C}\left(\theta\right)
\end{align}
where $\mathsf{C}\left(\theta\right)$ is called the \textit{group confidence} on the hypothesis $\theta$ and it is defined as
\begin{align*}
\mathsf{C}\left(\theta\right) & = -D_{KL}\left(\boldsymbol{f}\left(\cdot\right)\|\boldsymbol{l}\left(\cdot | \theta\right)\right)\\
& = -\sum\limits_{i=1}^n D_{KL}\left(f^i\left(\cdot\right)\|l^i\left(\cdot|\theta\right)\right).
\end{align*}

 Specifically, $D_{KL}\left(f^i\left(\cdot\right)\|l^i\left(\cdot|\theta\right)\right)$ is Kullback-Leibler divergence between the true distribution of $S_k^i$ and the probability distribution $l^i( \cdot | \theta)$ that would have been seen by agents $i$ if hypothesis $\theta$ were correct. The group confidence is the sum of the individual confidences for each of the agents. To avoid the trivial case we assume $\Theta^*$ set is a strict subset of $\Theta$. Agents do not know the distributions $f^i(\cdot)$ and they try to ``learn" the solution to this optimization problem based on local observations and interactions, see Figure~\ref{triangle}. 
\begin{figure}[h!]
	\centering
	\includegraphics[width=0.25\textwidth]{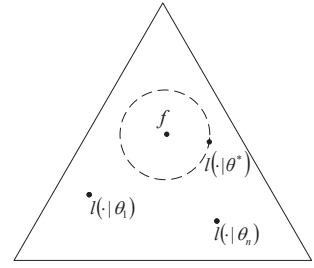}
	\caption{Geometric interpretation of the learning objective. The triangle represents the simplex composed of all agents' probability distributions. The observations of the agents are generated according to a joint probability distribution 
		$\boldsymbol{f}\left(\cdot\right)$. The joint distribution for the agent observations is parametrized by $\theta$. 
		The agent goal is to learn a hypothesis that best describes their observations, which corresponds to the distribution 
		$l(\cdot|\theta^*)$ (the closest to the distribution $\boldsymbol{f}\left(\cdot\right)$).}
	\label{triangle}
\end{figure}

The agents interact in a network modeled as a sequence of directed graphs $\{\mathcal{G}_k\}$. Each graph $\mathcal{G}_k = (V,E_k)$ is composed by a set of vertices (i.e. agents) $V$ and a set of directed links $E_k$, i.e. agent $i$ can only send messages to its out neighbors. Naturally we require some long term connectivity property for this graph sequence. We assume that the sequence  $\left\{\mathcal{G}_k\right\}$ is $B$-strongly-connected, i.e., there is an integer $B\ge 1$ such that the graph $\left(V,\bigcup_{i=kB}^{\left(k+1\right)B-1}E_i\right)$ is connected for all $k \geq 0$. Thus, we do not require every graph $\mathcal{G}_k$ to be connected instantaneously but rather over an uniform  period of time..

We propose a new belief update rule such that, as the number of observations increases, the beliefs concentrate in the set $\Theta^*$. Specifically we do so by showing that the beliefs for all $\theta \notin \Theta^*$ will go to zero. Furthermore, our main result states a non-asymptotic and geometric rate of convergence for the propose update rule. This converge rate is expressed explicitly in term of the number of agents, the network parameters as well as the group confidence for each of the hypothesis. Additionally this update rule is shown to induce a balanced behavior in the network, where independently on its connectivity, after a transient time all nodes will learn as if the sequence of graphs were balanced.

\subsection{Belief Update Rule}

We propose a new algorithm where every node update their beliefs on the hypothesis set following the next rule
\begin{subequations}\label{push_bayes}
\begin{align}
y^i_{k+1} & =\sum\limits_{j \in N_k^i }\frac{y_k^j}{d^j_k} \\
\mu_{k+1}^i\left(\theta\right) & = \frac{1}{Z_{k+1}^i}\left(\prod\limits_{j \in N_k^i }\mu_{k}^j\left(\theta\right)^{\frac{y_k^j}{d^j_k}}l^i\left(s_{k+1}^i|\theta\right)\right)^{\frac{1}{y_{k+1}^i}}
\end{align}
\end{subequations}
where at time $k$: $N_k^i$ is the set of in-neighbors of node $i$, that is $N_k^i = \{j|(j,i) \in E_k\}$ (a node is assumed to be its own neighbor) and the value $d^i_k$ its the out degree of node $i$. The term $Z_{k+1}^i$ is a normalization factor defined as,
\begin{align*}
 Z_{k+1}^i & = \sum\limits_{p=1}^{m}\left( \prod\limits_{j \in N_k^i }\mu_{k}^j\left(\theta_p\right)^{\frac{y_k^j}{d^j_k}}l^i\left(s_{k+1}^i|\theta_p\right)\right)^{\frac{1}{y_{k+1}^i}}. 
\end{align*}

The proposed update rule in Eqs. \eqref{push_bayes} is inspired by the Push-Sum protocol recently studied in \cite{kem03,ben10} and its application to distributed optimization in time-varying directed graphs \cite{tsi12,ned13,iut13,rab14,ger15}. At each time step, each node shares to its out neighbors its beliefs on the hypothesis set $\Theta$. Additionally, it also shares a self assigned weight $\frac{y_k^j}{d^j_k}$ which indicates how does it wants to be weighted by its neighbors. Node $i$ computes the geometric averages of the beliefs of its in-neighbor set. However node $i$ does not assign the self-weight node $j$ provided but a normalized version of it; on the contrary node $i$ adds all the self-weights as $y^i_{k+1} =\sum\limits_{j \in N_k^i }\frac{y_k^j}{d^j_k} $ and weights node $j$ with $\frac{y_k^j}{d^j_k}\frac{1}{y^i_{k+1}}$. At that point the term $y^i_{k+1}$ becomes the self-weight she will communicate to its out neighbors in the time step $k+1$. Once the geometric average step is done, then the update Bayesian step is performed based on the local observations with a learning rate parameter of $\frac{1}{y^i_{k+1}}$.

We require that if $f^i\left(s^i\right)>0$ there exists an $\alpha >0$ such that $l^i\left(s^i | \theta \right)  > \alpha$ for all $\theta \in \Theta$ for all agents $i=1,\ldots,n$. The corresponding constant $\alpha$ will appear in the bounds we will later derive on the convergence rate. Also, for each agent, there is a non-empty subset of the optimal hypothesis set on which the prior belief is strictly positive, i.e., there is a nonempty set $\Theta^{*i}\subseteq \Theta^*$ such that $\mu_{0}^i\left(\theta\right)>0$ for all  $\theta \in \Theta^{*i}$. Furthermore, the intersection set $\hat \Theta^*=\cap_{i=1}^n\Theta^{*i}$ is nonempty. This last assumption will be avoided by having prior beliefs with uniform distribution.

\subsection{Main Result}

Our main results provides the non-asymptotic convergence rates for the proposed update rule in Eq. \eqref{push_bayes}. This result explicitly states how the beliefs concentrate in the set of optimal hypothesis solutions to the optimization problem in Eq. \eqref{main_problem}. The proof for this theorem will be presented in Section \ref{rates}.

\begin{theorem}\label{thm_1}
	Let $\boldsymbol{f}$ be the distribution of the vector random variable $\boldsymbol{S_k}$ and suppose that:
	\begin{enumerate}[(a)]
		\item The sequence $\{\mathcal{G}_k\}$ is uniformly strongly connected.
		\item If $f^i\left(s^i\right)>0$, then there exists an $\alpha >0$ such that $l^i\left(s^i | \theta \right)  > \alpha$ for all $\theta \in \Theta$.
	\end{enumerate} 
	Also, let $\rho\in(0,1)$ be a given error percentile (or confidence value).
	Then, the update rule of Eqs.~(\ref{push_bayes}), with $y_0^i=1$ and uniform initial beliefs, has the following property:
	there is an integer $\boldsymbol{N}(\rho)$ such that, with probability $1 -\rho$, 
	for all $k\ge \boldsymbol{N}(\rho)$ and for all $\theta_v\notin\Theta^*$ 
	there holds
	$$\mu_{k}^i\left(\theta_v \right) \leq \exp\left( -\frac{k}{2} \gamma_2+ \frac{1}{\delta}\gamma_1^i\right)
	\quad\hbox{for all } i = 1, \ldots, n,$$ 
	where
	\begin{align*}
	\boldsymbol{N}(\rho)
	\triangleq \left\lceil\frac{8 \left( \log\left(\alpha \right)\right) ^2  \log\left(\frac{1}{\rho} \right) }{\delta^2\gamma_2^2} + 1 \right\rceil,
	\end{align*}
	\begin{align*}
	\gamma_1^i & = \max_{\substack{\theta_w\in \hat{\Theta}^* \\\theta_v \notin \Theta^*} }
	\left\{
	\frac{2C}{1- \lambda} \|\boldsymbol{H}\left(\theta_v,\theta_w\right)\|_1  - \left[ \boldsymbol{H}\left(\theta_v,\theta_w\right)\right]_i   \right\},\cr
	\gamma_2 &=  \frac{1}{n}\, \min_ {\theta_v\notin\Theta^*}
	\left(\mathsf{C}^*-\mathsf{C}\left(\theta_v\right)\right),
	\end{align*}
	with $\mathsf{C}\left(\theta\right)$ being the group confidence on hypothesis $\theta$ and $\mathsf{C}^* = \mathsf{C}(\theta)$ for all $\theta \in \Theta^*$ and the vector $\boldsymbol{H}\left(\theta_v,\theta_w\right)$ has coordinates given by
	\begin{align*}
	\left[ \boldsymbol{H}\left(\theta_v,\theta_w\right)\right]_i = 
	 D_{KL}(f^i(\cdot) \| l^i(\cdot|\theta_v)) - D_{KL}\left(f^i (\cdot) \|l^i\left(\cdot|\theta_w\right)\right). 
	\end{align*} 
	The constants $C$, $\delta$ and $\lambda$ satisfy the following relations:\\
	\noindent (1)
	For general $B$-strongly-connected graph sequences $\{\mathcal{G}_k\}$, 
	\begin{align*}
	C = 4, & \ \ \ \ \lambda = \left(1-\frac{1}{n^{nB}}\right)^{\frac{1}{B}}, \ \ \ \delta \geq \frac{1}{n^{nB}}. 
	\end{align*}
	\noindent (2) If every graph $G_k$ is regular with $B=1$,
	\begin{align*}
	C = \sqrt{2}, & \ \ \ \ \lambda = \left(1-\frac{1}{4n^3} \right)^{\frac{1}{B}}, \ \ \ \delta = 1.
	\end{align*} 
\end{theorem}

This theorem shows that the network of agents will collectively solve the optimization problem in Eq. \eqref{main_problem}. After a transient time $\boldsymbol{N}(\rho)$, the belief on the hypothesis outside the optimal hypothesis set that maximizes the group confidence will decay exponentially fast. Moreover, this will happen at a rate that depends on explicitly characterized terms $\gamma_1^i$ and $\gamma_2$. This exponential rate is network independent and hold for all the nodes in the network. Additionally, after a transient time of $\frac{2 \gamma_1^i}{\delta \gamma_2}$ for which the beliefs are bounded by $1$ the exponential decay will occur at a rate that depends on $\gamma_2$ only, i.e. the average difference between the optimal confidence and the second best hypothesis.

This result generalizes previously proposed algorithms  \cite{ned15} when the optimal set of hypothesis is also optimal from the local perspective \cite{ned14}. Moreover, in contrast with previous literature, the convergence rate induced by parameter $\gamma_2$ does not depend on the parameter $\delta$, that is, after a transient time the convergence rate is as if the sequence of graphs were regular. Without this regularization behavior the amount an agents contributes to the group confidence was determined by its location in the network, i.e. $\delta$. Then in the case of time-varying graphs the importance of the nodes might change as well, and since we allow for disjoint node optimal hypothesis, the concentration of the beliefs would oscillate with the confidence as a weighted sum of local confidences are changing with the topology of the network.  

\begin{rem}
	If the auxiliary sequence $y_k^i$ is not used in the update rule, i.e. 
	\begin{align*}
	\mu_{k+1}^i\left(\theta\right) & = \frac{1}{Z_{k+1}^i}\prod\limits_{j \in N_k^i }\mu_{k}^j\left(\theta\right)^{\frac{1}{d^j_k}}l^i\left(s_{k+1}^i|\theta\right)
	\end{align*}
	with the corresponding normalization term $Z_{k+1}^i$, we obtain a similar result as in Theorem \ref{thm_1} with the exponential rate
	$$\mu_{k}^i\left(\theta_v \right) \leq \exp\left( -\delta\frac{k}{2} \gamma_2+ \gamma_1^i\right)
	\quad\hbox{for all } i = 1, \ldots, n,$$ 
	with the same constants $\delta$, $C$, $\gamma_2$, $\gamma_1^i$ and $\boldsymbol{N}(\rho)$. However, after the same transient time $\frac{2 \gamma_1^i}{\delta \gamma_2}$, the exponential decay occurs at a rate that depends on  $\delta\gamma_2$. Where $\delta$ might be very small for highly unbalanced graphs; note that $\delta \geq \frac{1}{n^{nB}}$. This might slow down convergence considerably. No proof of this statement is provided due to space constraints, but the result is by using Lemma \ref{lemma_angelia} in Section \ref{rates}.
\end{rem}
	
\section{Convergence Rate Analysis}\label{rates}

In this section we analyze the dynamics of the proposed learning rule Eqs. \eqref{push_bayes}. First lets define the flowing quantities that simplifies the analysis procedure: for all $i=1,\hdots,n$ and $k \geq 0$ let
\begin{align}
\varphi_k^i\left(\theta_v,\theta_w\right) & \triangleq \log \frac{\mu_{k}^i\left(\theta_v\right)}{\mu_{k}^i\left(\theta_w\right)} \label{phi_def}\\
\hat{\varphi}_k^i\left(\theta_v,\theta_w\right) & \triangleq  y_{k}^i \varphi_k^i\left(\theta_v,\theta_w\right) \label{phi_def2}
\end{align}
for any $\theta_v \notin \hat{\Theta}^*$ and $\theta_w \in \hat{\Theta}^*$. With this definitions in place we can focus on analyzing the dynamics of $\hat{\varphi}_k^i\left(\theta_v,\theta_w\right)$. 

\begin{proposition}\label{prop_phi}
	The quantity $\hat{\varphi}_k^i\left(\theta_v,\theta_w\right)$ evolve as
	{\small
	\begin{align}\label{aux_phi}
	\hat{\varphi}_{k+1}^i\left(\theta_v,\theta_w\right) & = \sum\limits_{i=1}^{n} \left[A_k\right]_{ij} \hat{\varphi}_k^j\left(\theta_v,\theta_w\right) + \log \frac{l^i\left(s_{k+1}^i|\theta_v\right)}{l^i\left(s_{k+1}^i|\theta_w\right)}.
	\end{align}}
Moreover, by staking all $\hat{\varphi}_k^i\left(\theta_v,\theta_w\right)$ into a single vector, $\boldsymbol{\hat{\varphi}}_{k+1}\left(\theta_v,\theta_w\right)$ can be compactly stated as 
	\begin{align}\label{vector_phi}
	\boldsymbol{\hat{\varphi}}_{k+1}\left(\theta_v,\theta_w\right) & = A_k\boldsymbol{\hat{\varphi}}_{k}\left(\theta_v,\theta_w\right) + \mathcal{L}_{k+1}^{\theta_v,\theta_w}.
	\end{align}	
	where $A_k$ is a matrix such that
	\begin{align*}
	\left[A_k\right]_{ij} & = \begin{cases}
	\frac{1}{d_k^j} & \text{if } (j,i) \in E_k\\
	0 & \text{otherwise}
	\end{cases}
	\end{align*}
	and 
	\begin{align*}
	\left[\mathcal{L}_{k+1}^{\theta_v,\theta_w}\right]_i & = \log \frac{l^i\left(s_{k+1}^i|\theta_v\right)}{l^i\left(s_{k+1}^i|\theta_w\right)}
	\end{align*}
\end{proposition}
\begin{proof}
	By the definitions provided in the learning rule Eqs. \eqref{push_bayes} and Eqs. \eqref{phi_def} and \eqref{phi_def2} we have that
	\begin{align*}
		& \hat{\varphi}_{k+1}^i\left(\theta_v,\theta_w\right)  =  y_{k+1}^i \varphi_{k+1}^i\left(\theta_v,\theta_w\right) \\
		& \qquad =  y_{k+1}^i \log \frac{\mu_{k+1}^i\left(\theta_v\right)}{\mu_{k+1}^i\left(\theta_w\right)} \\
		&\qquad  = y_{k+1}^i \log \frac{\left( \prod\limits_{i=1}^{n}\mu_{k}^j\left(\theta_v\right)^{\left[A_k\right]_{ij}{y_k^j}}l^i\left(s_{k+1}^i|\theta_v\right)\right)^{\frac{1}{y_{k+1}^i}}}{\left( \prod\limits_{i=1}^{n}\mu_{k}^j\left(\theta_w\right)^{\left[A_k\right]_{ij}{y_k^j}}l^i\left(s_{k+1}^i|\theta_w\right)\right)^{\frac{1}{y_{k+1}^i}}} \\
		&\qquad  = \log \frac{\left( \prod\limits_{i=1}^{n}\mu_{k}^j\left(\theta_v\right)^{\left[A_k\right]_{ij}{y_k^j}}l^i\left(s_{k+1}^i|\theta_v\right)\right)}{\left( \prod\limits_{i=1}^{n}\mu_{k}^j\left(\theta_w\right)^{\left[A_k\right]_{ij}{y_k^j}}l^i\left(s_{k+1}^i|\theta_w\right)\right)} \\
		&\qquad  = \sum\limits_{i=1}^{n} \left[A_k\right]_{ij} y_k^j \log \frac{\mu_{k}^j\left(\theta_v\right)}{\mu_{k}^j\left(\theta_w\right)}+ \log \frac{l^i\left(s_{k+1}^i|\theta_v\right)}{l^i\left(s_{k+1}^i|\theta_w\right)} \\
		&\qquad  = \sum\limits_{i=1}^{n} \left[A_k\right]_{ij} \hat{\varphi}_k^j\left(\theta_v,\theta_w\right) + \log \frac{l^i\left(s_{k+1}^i|\theta_v\right)}{l^i\left(s_{k+1}^i|\theta_w\right)}.
	\end{align*}
	
	The first three equalities follow from  Eq. \eqref{push_bayes}, \eqref{phi_def} and \eqref{phi_def2}. Cancellation of the term $y_{k+1}^i$ leads to the fourth equality. The rest of the proof follows from arithmetic properties of logarithms.
\end{proof}

With this result in hand we are able to further analyze the sequence $\hat{\varphi}_{k+1}^i\left(\theta_v,\theta_w\right)$. First by adding and subtracting the term $\sum\limits_{t=1}^{k} \phi_k \mathbf{1}' \mathcal{L}_{t}^{\theta_v,\theta_w}$ from Eq. \eqref{vector_phi} we obtain
\begin{align*}
\boldsymbol{\hat{\varphi}}_{k+1}\left(\theta_v,\theta_w\right) & = A_{k:0}\boldsymbol{\hat{\varphi}}_{0}\left(\theta_v,\theta_w\right) + \sum\limits_{t=1}^{k}A_{k:t}  \mathcal{L}_{t}^{\theta_v,\theta_w} + \mathcal{L}_{k+1}^{\theta_v,\theta_w} \\
& \qquad  - \sum\limits_{t=1}^{k} \phi_k \mathbf{1}' \mathcal{L}_{t}^{\theta_v,\theta_w} +\sum\limits_{t=1}^{k} \phi_k \mathbf{1}' \mathcal{L}_{t}^{\theta_v,\theta_w} \\
& =  A_{k:0}\boldsymbol{\hat{\varphi}}_{0}\left(\theta_v,\theta_w\right) + \sum\limits_{t=1}^{k} D_{k:t}   \mathcal{L}_{t}^{\theta_v,\theta_w} +  \mathcal{L}_{k+1}^{\theta_v,\theta_w}\\
& \qquad  +\sum\limits_{t=1}^{k} \phi_k \mathbf{1}' \mathcal{L}_{t}^{\theta_v,\theta_w}
\end{align*}
with $D_{k:t}  = A_{k:t} - \phi_k \mathbf{1}'$.

From now on we will ignore the first term in $\boldsymbol{\hat{\varphi}}_{k+1}\left(\theta_v,\theta_w\right)$, assuming all agents use a uniform distribution as their initial beliefs, thus $\hat{\varphi}^i_{0}\left(\theta_v,\theta_w\right) =0$. This simplifies the notation and facilitates the exposition of the results, moreover, it does not limit the generality of our method since this term can be upper bounded and it will depend at most linearly with the number of agents.

Now by going back from $\boldsymbol{\hat{\varphi}}_k\left(\theta_v,\theta_w\right)$ to $\boldsymbol{\varphi}_k\left(\theta_v,\theta_w\right)$ we have that
\begin{align*}
&\varphi_{k+1}^i\left(\theta_v,\theta_w\right) = \\ 
& \qquad  \frac{\sum\limits_{t=1}^{k} \left[ D_{k:t}  \mathcal{L}_{t}^{\theta_v,\theta_w} \right]_i +  \left[ \mathcal{L}_{k+1}^{\theta_v,\theta_w}\right]_i +\sum\limits_{t=1}^{k} \phi_k^i \mathbf{1}' \mathcal{L}_{t}^{\theta_v,\theta_w}}{y_{k+1}^i}
\end{align*}
Similarly the dynamics of $\boldsymbol{y}_k$ can be expressed as
\begin{align*}
\boldsymbol{y}_{k+1} & = A_{k:0}\boldsymbol{y}_0 \\
& = A_{k:0}y_0 - \phi_k \mathbf{1}' \boldsymbol{y}_0 + \phi_k \mathbf{1}'\boldsymbol{y}_0 \\
& = D_{k:0}\mathbf{1}+ \phi_k n
\end{align*}
which leads us to
\begin{align}\label{phi_iter}
& \varphi_{k+1}^i\left(\theta_v,\theta_w\right) =  \nonumber \\
& \qquad \frac{ \sum\limits_{t=1}^{k} \left[ D_{k:t}  \mathcal{L}_{t}^{\theta_v,\theta_w} \right]_i +  \left[ \mathcal{L}_{k+1}^{\theta_v,\theta_w}\right]_i +\sum\limits_{t=1}^{k} \phi_k^i \mathbf{1}' \mathcal{L}_{t}^{\theta_v,\theta_w}}{\left[ D_{k:0}\mathbf{1}\right] _i+ \phi_k^i n}
\end{align}

The next lemma will provide a general tool for analyzing the non-asymptotic properties of a learning rule that can be expressed as a log-linear function of bounded variations and upper bounded expectation as it was recently used in \cite{ned14,ned15}. It can be interpreted as a specialized version of the McDiarmid concentration \cite{mcd89} for log-linear update rules.
\begin{lemma}\label{mc_bayes}
	Consider a learning update rule that can be expressed as a log-linear function, i.e.,
	\begin{align*}
	\mu_{k+1}^i\left(\theta_v\right) & \leq \exp \left( \varphi^i_{k+1}\left(\theta_v,\theta_w\right)\right). 
	\end{align*}
	If the term  $ \varphi^i_{k+1}\left(\theta\right)$ is of bounded variations with bounds $\{c^i_k\}$ at each time $k$ and its expected value is upper bounded by an affine function as ${\mathbb{E}\left[\varphi^i_{k+1}\left(\theta\right)\right] \leq \frac{1}{\delta}\gamma_1^i - k \gamma_2}$
	Then,
	{\small
	\begin{align*}
	\mathbb{P}\left(\mu_{k+1}^i\left(\theta_v \right) 
	\geq \exp\left( -\frac{k}{2}\gamma_2 + \frac{1}{\delta}\gamma_1^i \right) \right)  & \leq   \exp\left( -\frac{\frac{1}{2} \left( k\gamma_2 \right)^2}{\sum_{t=1}^{k+1}\left(  c_t^i\right)  ^2} \right)
	\end{align*}
	}
\end{lemma} 
\begin{proof}
	Following simple set properties of the probability measure on the desired set $\mu_{k+1}^i\left(\theta_v \right) 
	\geq \exp\left( -\frac{k}{2}\gamma_2 + \frac{1}{\delta}\gamma_1^i \right) $ we have that,
	\begin{align*}
	& \mathbb{P}\left(\mu_{k+1}^i\left(\theta_v \right) 
	\geq \exp\left( -\frac{k}{2}\gamma_2 + \frac{1}{\delta}\gamma_1^i \right) \right) \\
	& \qquad \leq  \mathbb{P}\left(\exp\left( \varphi_{k+1}^i\left(\theta_v,\theta_w\right) \right)   
	\geq \exp\left( -\frac{k}{2}\gamma_2 + \frac{1}{\delta}\gamma_1^i  \right) \right) \\
	& \qquad =  \mathbb{P}\left(\varphi_{k+1}^i\left(\theta_v,\theta_w\right) \geq   -\frac{k}{2}\gamma_2 + \frac{1}{\delta}\gamma_1^i \right)  \\
	& \qquad  =  \mathbb{P}\left(\varphi_{k+1}^i\left(\theta_v,\theta_w\right) - \mathbb{E}\left[\varphi_{k+1}^i\left(\theta_v,\theta_w\right) \right]  \geq  \right. \\
	& \qquad \qquad \qquad \left.  -\frac{k}{2}\gamma_2 + \frac{1}{\delta}\gamma_1^i  - \mathbb{E}\left[\varphi_{k+1}^i\left(\theta_v,\theta_w\right) \right] \right) \\
	& \qquad  = \mathbb{P}\left(\varphi_{k+1}^i\left(\theta_v,\theta_w\right) - \mathbb{E}\left[\varphi_{k+1}^i\left(\theta_v,\theta_w\right) \right]  \geq   \frac{k}{2}\gamma_2  \right).
	\end{align*}
	
	Use McDiarmid's inequality to get the desired result.
\end{proof}

Before proceeding with the analysis of the learning rule we will recall a result from \cite{ned13} about the geometric rates of convergence of product of column stochastic matrices.

\begin{lemma}\label{lemma_angelia}
	[Corollary 2.a in \cite{ned13}] Let the graph sequence $\{\mathcal{G}_k\}$, with $\mathcal{G}_k = \left(E_k,V\right)$ be uniformly strongly connected. Then, there is a sequence $\{\phi_k\}$ of stochastic vectors such that,
	\begin{align*}
	|\left[A_{k:t}\right]_{ij} - \phi_k^i| & \leq C \lambda^{k-t}  \ \ \ \ \ \ \text{for all } \ k \leq t \leq 0
	\end{align*}
	for $C$ and $\lambda \in \left(0,1\right)$ as described in Theorem \ref{thm_1}.
\end{lemma}
\begin{lemma} \label{lemma_deltabound}
	[Corollary 2.b in \cite{ned13}] Let the graph sequence $\left\{\mathcal{G}_k\right\}$ satisfy the B-strong connectivity assumption. Define
	\begin{align}\label{eq:defdelta}
	\delta \triangleq\inf_{k\geq 0} \left(\min_{1\leq i\leq n}\left[A_{k:0} \mathbf{1}_n \right]_i\right).
	\end{align}
	Then, $\delta \geq \frac{1}{n^{nB}}$, and if all $\mathcal{G}_k$ with $B=1$ are regular, then $\delta=1$.
	Furthermore, the sequence
	$\phi_k$ from Lemma \ref{lemma_angelia} satisfies $\phi_k^j \geq \delta/n$ for all $t \geq 0, j = 1, \ldots, n$.
\end{lemma}

Now the next lemma and proposition will show the desired properties required in lemma \ref{mc_bayes} to get the non-asymptotic results. First we will show the bounds on the expected value and then the bounded variation property.   

\begin{lemma}\label{lemma_expected}
	Consider $\varphi_{k+1}^i\left(\theta_v,\theta_w\right)$ as defined in Eq. \eqref{phi_def} then
\begin{align*}
\mathbb{E}\left[\varphi_{k+1}^i\left(\theta_v,\theta_w\right)\right] \leq \frac{1}{\delta}\gamma_1^i - k\gamma_2
\end{align*}
for all $i$ and $k\geq0$, where
\begin{align*}
\gamma_1^i & = \max_{\substack{\theta_w\in \hat{\Theta}^* \\\theta_v \notin \Theta^*} }
\left\{
\frac{2C}{1- \lambda} \|\boldsymbol{H}\left(\theta_v,\theta_w\right)\|_1  - \left[ \boldsymbol{H}\left(\theta_v,\theta_w\right)\right]_i   \right\},\cr
\gamma_2 &=  \frac{1}{n}\, \min_ {\theta_v\notin\Theta^*}
\left(\mathsf{C}^*-\mathsf{C}\left(\theta_v\right)\right),
\end{align*}
\end{lemma}

\begin{proof}
	First by taking the expected value of Eq. \eqref{phi_iter} we have that for all $k \geq 0$,
	{\small
	\begin{align*}
		& \mathbb{E}\left[\varphi_{k+1}^i\left(\theta_v,\theta_w\right)\right] \\
		& = \frac{ \sum\limits_{t=1}^{k} \left[ D_{k:t}  \boldsymbol{H}\left(\theta_v,\theta_w\right) \right]_i +  \left[  \boldsymbol{H}\left(\theta_v,\theta_w\right)\right]_i +\sum\limits_{t=1}^{k} \phi_k^i \mathbf{1}'  \boldsymbol{H}\left(\theta_v,\theta_w\right)}{\left[ D_{k:0}\mathbf{1}\right] _i+ \phi_k^i n} \\
		& = \frac{ \sum\limits_{t=1}^{k} \left[ D_{k:t}  \boldsymbol{H}\left(\theta_v,\theta_w\right) \right]_i +  \left[  \boldsymbol{H}\left(\theta_v,\theta_w\right)\right]_i + k \phi_k^i \mathbf{1}'  \boldsymbol{H}\left(\theta_v,\theta_w\right)}{\left[ D_{k:0}\mathbf{1}\right] _i+ \phi_k^i n} 
	\end{align*}
}
The main idea is to analyze how the term $\mathbb{E}\left[\varphi_{k+1}^i\left(\theta_v,\theta_w\right)\right]$ differs from a dynamic term where all agents have the same importance in the network and thus the learning occurs at a rate $\frac{\mathbf{1}'\boldsymbol{H}\left(\theta_v,\theta_w\right)}{n} $.

As a first step lets add and subtract the term $k\frac{\mathbf{1}'\boldsymbol{H}\left(\theta_v,\theta_w\right)}{n} $, therefore we obtain
\begin{align*}
& \mathbb{E}\left[\varphi_{k+1}^i\left(\theta_v,\theta_w\right)\right] \\
 & = \frac{ \sum\limits_{t=1}^{k} \left[ D_{k:t}  \boldsymbol{H}\left(\theta_v,\theta_w\right) \right]_i +  \left[  \boldsymbol{H}\left(\theta_v,\theta_w\right)\right]_i +k \phi_k^i \mathbf{1}'  \boldsymbol{H}\left(\theta_v,\theta_w\right)}{\left[ D_{k:0}\mathbf{1}\right] _i+ \phi_k^i n} \\
& \ \ \ \ \ \ \ 
- k\frac{\mathbf{1}'\boldsymbol{H}\left(\theta_v,\theta_w\right)}{n}
+ k\frac{\mathbf{1}'\boldsymbol{H}\left(\theta_v,\theta_w\right)}{n}
\end{align*}
by working out the arithmetics we have
{\small
\begin{align*}
& \mathbb{E}\left[\varphi_{k+1}^i\left(\theta_v,\theta_w\right)\right] \\
& = \frac{ n \left( \sum\limits_{t=1}^{k} \left[  D_{k:t} \boldsymbol{H}\left(\theta_v,\theta_w\right) \right]_i + \left[  \boldsymbol{H}\left(\theta_v,\theta_w\right)\right]_i + k \phi_k^i  \mathbf{1}'\boldsymbol{H}\left(\theta_v,\theta_w\right) \right)  }{n\left(  \left[ D_{k:0}\mathbf{1}\right] _i+ \phi_k^i n\right) }  
 \\ 
& \qquad -\frac{\left( \left[ D_{k:0}\mathbf{1}\right] _i+ \phi_k^i n\right) k\mathbf{1}'\boldsymbol{H}\left(\theta_v,\theta_w\right)}{n\left(  \left[ D_{k:0}\mathbf{1}\right] _i+ \phi_k^i n\right) } +k\frac{\mathbf{1}'\boldsymbol{H}\left(\theta_v,\theta_w\right)}{n} \\
& = \frac{ n \left( \sum\limits_{t=1}^{k} \left[  D_{k:t} \boldsymbol{H}\left(\theta_v,\theta_w\right) \right]_i + \left[  \boldsymbol{H}\left(\theta_v,\theta_w\right)\right]_i \right)  }{n\left(  \left[ D_{k:0}\mathbf{1}\right] _i+ \phi_k^i n\right) } \\
& \qquad  -\frac{ \left( \left[ D_{k:0}\mathbf{1}\right] _i\right) k\mathbf{1}'\boldsymbol{H}\left(\theta_v,\theta_w\right)}{n\left(  \left[ D_{k:0}\mathbf{1}\right] _i+ \phi_k^i n\right) } +k\frac{\mathbf{1}'\boldsymbol{H}\left(\theta_v,\theta_w\right)}{n} 
\end{align*}
}
Before finalizing the proof note that the denominator of the above function has the property ${\left[ D_{k:0}\mathbf{1}\right]_i+ \phi_k^i n > \delta}$. As noted in \cite{ned13}, this follows from the fact that this term is the sum of the $i$-th row of the matrix $A_{k:0}$ multiplied $n$ times. Therefore by taking absolute value of the first terms we obtain,
{\small
\begin{align*}
& \mathbb{E}\left[\varphi_{k+1}^i\left(\theta_v,\theta_w\right)\right] \\
& \leq \frac{1}{\delta}\left( \sum\limits_{t=1}^{k}\left(\max_j |\left[  D_{k:t} \right]_{ij} | \right)  \|\boldsymbol{H}\left(\theta_v,\theta_w\right)\|_1 
+ \left[  \boldsymbol{H}\left(\theta_v,\theta_w\right)\right]_i  \right) 
\\
& \qquad + \frac{k}{n \delta}  \|\boldsymbol{H}\left(\theta_v,\theta_w\right)\|_1\left(\max_j |\left[  D_{k:0} \right]_{ij} | \right)n  + k\frac{\mathbf{1}'\boldsymbol{H}\left(\theta_v,\theta_w\right)}{n} 
\end{align*}
}
Using Lemma \ref{lemma_angelia} we obtain upper bounds on $|\left[  D_{k:t} \right]_{ij}|$ where 
{\small
\begin{align*}
& \mathbb{E}\left[\varphi_{k+1}^i\left(\theta_v,\theta_w\right)\right] \\
& \leq \frac{1}{\delta}\left( C \sum\limits_{t=1}^{k} \lambda^{k-t}  \|\boldsymbol{H}\left(\theta_v,\theta_w\right)\|_1 
+ \left[  \boldsymbol{H}\left(\theta_v,\theta_w\right)\right]_i  \right) \\
& \qquad
+ k\frac{C}{\delta} \lambda^t  \|\boldsymbol{H}\left(\theta_v,\theta_w\right)\|_1+ k\frac{\mathbf{1}'\boldsymbol{H}\left(\theta_v,\theta_w\right)}{n} \\
& \leq \frac{2C}{\delta}\frac{1}{1- \lambda} \|\boldsymbol{H}\left(\theta_v,\theta_w\right)\|_1  +  \frac{1}{\delta}\left[  \boldsymbol{H}\left(\theta_v,\theta_w\right)\right]_i   
+ k\frac{\mathbf{1}'\boldsymbol{H}\left(\theta_v,\theta_w\right)}{n} 
\end{align*}
}
Finally note that we can express the term $\mathbf{1}'\boldsymbol{H}\left(\theta_v,\theta_w\right)$ in term of the group confidence as ${\mathbf{1}'\boldsymbol{H}\left(\theta_v,\theta_w\right) = \mathsf{C}^* -\mathsf{C}\left(\theta_v\right)}$. 
The desired bound will follow by defining,
\begin{align*}
\gamma_1^i & =  \max_{\substack{\theta_w\in \hat{\Theta}^* \\\theta_v \notin \Theta^*} }
\left\{ 
\frac{2C}{1- \lambda} \|\boldsymbol{H}\left(\theta_v,\theta_w\right)\|_1  -  \left[ \boldsymbol{H}\left(\theta_v,\theta_w\right)\right]_i  \right\},\cr
\gamma_2 &=  \frac{1}{n}\, \min_ {\theta_v\notin\Theta^*}
\left(\mathsf{C}^*-\mathsf{C}\left(\theta_v\right)\right).
\end{align*}
\end{proof}

Next we will show that the term $\varphi_{k}^i\left(\theta_v,\theta_w\right)$, as a function of a sequence of $t$ random vectors, is of bounded variations.  

\begin{proposition}\label{prop_bounded}
	The term  $\varphi_{k}^i\left(\theta_v,\theta_w\right)$ is a function of a sequence of $t$ random vectors and it is of bounded variations.
\end{proposition}
\begin{proof}
	Following the definition of a bounded variation function we have,
	\begin{align*}
	\max_{\textbf{s}_t \in \mathcal{S}} \varphi_{k+1}^i\left(\theta_v,\theta_w\right) 	& =  \max_{\textbf{s}_t \in \mathcal{S}} 
	\frac{\sum_{j=1 }^{n}\left[  A_{k:t}\right] _{ij} \left[ \mathcal{L}^{\theta_v,\theta_w}_t \right]_j}{y_{k+1}^i}   \\
	& =  \max_{\textbf{s}_t \in \mathcal{S}} 
	\frac{\sum_{j=1 }^{n}\left[  A_{k:t}\right] _{ij} \frac{y_t^j}{y_t^j} \left[ \mathcal{L}^{\theta_v,\theta_w}_t \right]_j}{y_{k+1}^i}  
	\end{align*}
	Now we can define a new set of weights $\left[B_{k:t}\right]_{ij} = \frac{\left[A_{k:t}\right]_{ij}y_t^j}{y_{k+1}^i} = \frac{\left[A_{k:t}\right]_{ij}y_t^j}{\sum\limits_{i=1}^{n}\left[A_{k:t}\right]_{ij}y_t^j}$ that multiply the vector whose entries are $\frac{\left[ \mathcal{L}^{\theta_v,\theta_w}_t \right]_j}{y_t^j}$. This new set of weights add up to one, i.e. $\sum\limits_{j=1}^{n}\left[B_{k:t}\right]_{ij} =1$ thus 
	\begin{align*}
	\max_{\textbf{s}_t \in \mathcal{S}} \varphi_{k+1}^i\left(\theta_v,\theta_w\right) &  =  \max_{\textbf{s}_t \in \mathcal{S}}  \sum_{j=1}^{n}\left[  B_{k:t}\right] _{ij} \frac{\left[ \mathcal{L}^{\theta_v,\theta_w}_t \right]_j}{y_t^j}   \\
	%& \leq \max_{s_t^i \in \mathcal{S}^i} \log \frac{l^i\left(s_t^i|\theta \right) }{l^i\left(s_t^i|\theta^* \right)} -\min_{s_t^i \in \mathcal{S}^i} \sum_{i=1}^n \log \frac{l^i\left(s_t^i|\theta \right) }{l^i\left(s_t^i|\theta^* \right)} \\
	& \leq \frac{1}{\delta}\left( \log \frac{1 }{\alpha} \right) 
	\end{align*}
	Finally, by symmetry we also have that 
	\begin{align*}
	-\min_{\textbf{s}_t \in \mathcal{S}} \varphi_{k+1}^i\left(\theta_v,\theta_w\right) & \leq \frac{1}{\delta}\left( \log \frac{1 }{\alpha} \right)
	\end{align*}
	Therefore,
	\begin{align*}
		\max_{\textbf{s}_t \in \mathcal{S}} \varphi_{k+1}^i\left(\theta_v,\theta_w\right)-\min_{\textbf{s}_t \in \mathcal{S}} \varphi_{k+1}^i\left(\theta_v,\theta_w\right) & \leq \frac{2}{\delta}\left( \log \frac{1 }{\alpha} \right) 
	\end{align*}
	This completes the proof.
\end{proof}

At this point we are ready to proof the main result.

\begin{proof}
	[Theorem \ref{thm_1}]
	The proof procedure will be a compilation of previous Lemmas and propositions. As a first step we will show that the proposed learning rule can be expressed as a log-linear function.
	
	Since $\mu_k^i\left(\theta\right) \in (0,1]$ for all $i =1, \hdots,n$, $k \geq 0$ and all $\theta \in \Theta$, we have that,
	{\small
	\begin{align*}
	& \mu_{k+1}^i\left(\theta_v\right)  \leq \frac{\mu_{k+1}^i\left(\theta_v\right)}{\mu_{k+1}^i\left(\theta_w\right)} \\
	& = \frac{\left( \prod\limits_{i=1}^{n}\mu_{k}^j\left(\theta_v\right)^{\left[A_k\right]_{ij}{y_k^j}}l^i\left(s_{k+1}^i|\theta_v\right)\right)^{\frac{1}{y_{k+1}^i}}}{\left( \prod\limits_{i=1}^{n}\mu_{k}^j\left(\theta_w\right)^{\left[A_k\right]_{ij}{y_k^j}}l^i\left(s_{k+1}^i|\theta_w\right)\right)^{\frac{1}{y_{k+1}^i}}} \\
	& = \exp \left(\frac{1}{y_{k+1}^i} \left( \sum\limits_{i=1}^{n} \left[A_k\right]_{ij}{y_k^j} \log \frac{\mu_{k}^j\left(\theta_v\right)}{\mu_{k}^j\left(\theta_w\right)} + \log \frac{l^i\left(s_{k+1}^i|\theta_v\right)}{l^i\left(s_{k+1}^i|\theta_w\right)} \right) \right) \\
	& = \exp \left(\frac{1}{y_{k+1}^i} \left( \sum\limits_{i=1}^{n} \left[A_k\right]_{ij} \hat{\varphi}_k^j\left(\theta_v,\theta_w\right) + \log \frac{l^i\left(s_{k+1}^i|\theta_v\right)}{l^i\left(s_{k+1}^i|\theta_w\right)} \right) \right) \\
	& = \exp \left(\varphi_{k+1}^j\left(\theta_v,\theta_w\right) \right)
	\end{align*} 
	}
	This result along side lemma \ref{lemma_expected} and Proposition \ref{prop_bounded} provides the conditions for Lemma \ref{mc_bayes}, thus the following relation is valid,
	{\small
	\begin{align*}
	\mathbb{P}\left(\mu_{k+1}^i\left(\theta_v \right) 
	\geq \exp\left( -\frac{k}{2}\gamma_2 + \gamma_1^i \right) \right)  & \leq   \exp\left( -\frac{\frac{1}{2} \left( k\gamma_2 \right)^2}{\sum_{t=1}^{k+1}\left(  c_t^i\right)  ^2} \right)
	\end{align*}
	}
	specifically, we have that $c_t^i = \frac{2}{\delta} \log \frac{1}{\alpha}$. Therefore
\begin{align*}
& \mathbb{P}\left(\varphi_{k+1}^i\left(\theta_v,\theta_w\right) - \mathbb{E}\left[\varphi_{k+1}^i\left(\theta_v,\theta_w\right) \right]  
\geq   \frac{k}{2}\gamma_2  \right) \\
& \qquad \leq   \exp\left( -\frac{\frac{1}{2} \left( k\gamma_2 \right)^2}{\sum_{t=1}^{k+1} \left( \frac{2}{\delta} \log \frac{1}{\alpha}\right) ^2} \right)  \\
& \qquad =  \exp\left( -\frac{\left( k\gamma_2 \delta^2 \right)^2}{ 8 (k+1) \left( \log \frac{1}{\alpha}\right)^2}  \right) \\
& \qquad \leq   \exp\left( -\frac{ (k-1)\gamma_2 ^2\delta^2 }{ 8  \left(  \log \alpha\right)^2} \right). 
\end{align*}
Finally, for a given confidence level $\rho$, in order to have 
$\mathbb{P}\left(\mu_{k}^i\left(\theta_v \right) \geq \exp\left( -\frac{1}{2}k\gamma_2 + \gamma_1^i \right) \right) 
\leq \rho$ we require that
\begin{align*}
k & \geq \frac{8 \left( \log\left(\alpha \right)\right) ^2  \log\frac{1}{\rho} }{\delta^2 \gamma_2^2} +1.
\end{align*}
This completes the proof.
\end{proof}

\section{Conclusions and Future Work}\label{conclusions}

We proposed a new update rule for the problem of distributed non-Bayesian learning on time-varying directed graphs with conflicting hypothesis. We show that the beliefs of all agents concentrate around a optimal set of hypothesis explicitly characterized as the solution to an optimization problem. This optimization problem consists on finding a probability distribution (from a parametrized family of distributions) closest to the unknown distribution of the observations and it needs to be solved by the agents interacting over a sequence of network and using local information only. The proposed algorithm also guarantees that after a finite transient time, that depends on the network structure, all agents will learn at a network independent rate that is the average of the agents individual learning abilities. We refer this as a ``balanced" behavior since all agents are weighted equally even if its connectivity is different. This results guarantees certain robustness properties of the learning process since faulty sensors or adversarial agents will not have any vantage even if they are centrally located.

Further research is required to study the case of a continuum set of hypothesis and the efficient transmission of probability distributions. Connections of the proposed algorithm and distributed balancing of matrices need to be explored as well. Furthermore, the characterization of the distributed non-Bayesian update rules as the solution of well defined optimization procedure generates a plethora of approaches to efficiently solve the estimation problem in many scenarios. 

\bibliographystyle{IEEEtran} %using IEEEtr was causing me trouble - Angelia

\bibliography{IEEEfull,bayes_cons_2}

\end{document}